\documentclass[a4paper,11pt]{amsart}

\usepackage{amssymb,xspace,amscd,yhmath,mathdots,txfonts,
mathrsfs}

\usepackage[matrix,arrow,curve]{xy}\CompileMatrices

\usepackage{hyperref}

\usepackage{yfonts}[1998/10/03]

\DeclareMathAlphabet{\mathpzc}{OT1}{pzc}{m}{it}

\numberwithin{equation}{section}

\theoremstyle{plain}

\newtheorem{thm}{Theorem}[section]

\newtheorem{lem}[thm]{Lemma}

\newtheorem{cor}[thm]{Corollary}

\newtheorem{prop}[thm]{Proposition}

\theoremstyle{definition}
\newtheorem{dfn}{Definition}[section]
\newtheorem{defn}{Definition}[section]

\newtheorem*{ntz}{Notation}

\DeclareMathAlphabet{\mathpzc}{OT1}{pzc}{m}{it}

\DeclareMathOperator{\R}{{\mathbb{R}}}
\DeclareMathOperator{\Ub}{\mathbf{U}}

\DeclareMathOperator{\kt}{{\kappaup}}

\newcommand\Levi{\mathrm{L}}

\DeclareMathOperator{\Gf}{\mathbf{G}}

\DeclareMathOperator{\ad}{\mathrm{ad}}
\DeclareMathOperator{\Ad}{\mathrm{Ad}}
\DeclareMathOperator{\Lie}{\mathrm{Lie}}

\DeclareMathOperator{\zt}{\mathfrak{z}}

\DeclareMathOperator{\pt}{\mathfrak{p}}

\DeclareMathOperator{\Ind}{\mathrm{Int}}

\DeclareMathOperator{\C}{\mathbb{C}}

\newcommand\Kf{\mathbf{K}}
\newcommand\kil{\mathit{b}}
\newcommand\Sf{\mathbf{E}}
\newcommand\sta{\mathfrak{stab}}
\newcommand\Tf{\mathbf{T}}

\newcommand\Gfu{\mathbf{G}_{u}}
\newcommand\gtu{\mathfrak{g}_u}
\newcommand\GL{\mathbf{GL}}

\newcommand\Wg{\mathpzc{W}}
\newcommand\ttu{\mathfrak{t}_u}
\newcommand\qt{\mathfrak{q}}
\newcommand\gt{\mathfrak{g}}

\newcommand\Rad{\mathpzc{R}}
\newcommand\hg{\mathfrak{h}}

\newcommand\vt{\mathfrak{v}}
\newcommand\eg{\mathfrak{e}}

\newcommand\Vf{\mathbf{V}}

\newcommand\mt{\textswab{m}}
\newcommand\Zf{\mathpzc{Z}}

\newcommand\SL{\mathbf{SL}}
\newcommand\SU{\mathbf{SU}}
\newcommand\CP{\mathbb{CP}}
\newcommand\Gr{\mathpzc{Gr}}

\newcommand\Bf{\mathbf{B}}
\newcommand\bt{\mathfrak{b}}

\newcommand\nt{\mathfrak{n}}

\newcommand\Mf{\mathpzc{M}}

\newcommand{\Ef}{\mathbf{E}}

\newcommand\slt{\mathfrak{sl}}

\newcommand\HNR{$\mathrm{HNR}$}

\newcommand\Qf{\mathbf{Q}}

\newcommand\fg{\mathfrak{f}}
\newcommand\Fb{\mathbf{F}}
\newcommand\hf{\mathpzc{h}}

   \def\DHLhksqrt#1#2{\setbox0=\hbox{$#1\sqrt{#2\,}$}\dimen0=\ht0
     \advance\dimen0-0.2\ht0
     \setbox2=\hbox{\vrule height\ht0 depth -\dimen0}%
     {\box0\lower0.4pt\box2}}

\title[Orbits of real forms and Matsuki duality]
{Orbits of real forms, Matsuki duality and 
{\texorpdfstring{$CR$}{TEXT}-cohomology}}

\author[S.Marini]{Stefano Marini}
\address{S.~Marini: Dipartimento di Matematica e Fisica, 
III Universit\`a di Roma, 
Largo San Leonardo Murialdo, 1
00146, Roma (Italy) }
\email{marinistefano86@gmail.com}
\author[M.~Nacinovich]{Mauro Nacinovich}
\address{M.\ Nacinovich:
Dipartimento di Matematica\\ II Universit\`a di Roma
``Tor Ver\-ga\-ta''\\ Via della Ricerca Scientifica\\ 00133 Roma
(Italy)}
\email{nacinovi@mat.uniroma2.it}

\date{November 29, 2016}


\begin{document}
%
\maketitle
\tableofcontents
This paper gives an overview on some topics concerning compact homogeneous
$CR$ manifolds,  
that have been developed in the last few years. 
The algebraic structure of 
compact Lie group  was employed
in \cite{AMN2013} to show that a large class of
compact $CR$ manifolds can be viewed as the total spaces of fiber 
bundles over  complex flag manifolds,
generalizing the classical Hopf fibration for the odd
dimensional sphere and the
Boothby-Wang fibration for homogeneous 
compact contact manifolds (see \cite{BW58}).
If a compact group $\Kf_0$  acts as a transitive group
of $CR$ diffeomorphisms of a $CR$ manifold $M_0,$ which
 is $\nt$-reductive in the sense of \cite{AMN2013},
one can construct a homogeneous space $M_-=\Kf/\Vf$ of the 
complexification $\Kf$ of $\Kf_0$ such that the map $M_0\to{M}_-$
associated to the inclusion $\Kf_0\hookrightarrow\Kf$
is a generic $CR$ embedding. The manifold $M_-$ is algebraic
over $\C$ and a tubular neighborhood of $M_0.$ 
For instance, if $M_0$ is the sphere $S^{2n-1}$ in $\C^n$ and $\Kf_0=\SU(n),$
the embedding $M_0\hookrightarrow{M}=\CP^{n}$  
is useful to compute the maximal group of $CR$ automorphisms of $M_0$
(see \cite{Cart1932a,Cart1932b}), while the embedding $M_0\hookrightarrow{M}_-=\C^n\setminus\{0\}$
better reflects the topology and the \textit{$CR$ cohomology} of $M_0.$ 
Thus, for some aspects of $CR$ geometry, we can consider 
$M_-$ to be  the \textit{best} complex realization of $M_0.$ This 
is the essential contents of the PhD thesis of the first Author
(\cite{Marini}): his aim was to show that, in 
a range which depends on the pseudoconcavity
of $M_0,$ the groups of 
tangential Cauchy-Riemann cohomology of $M_0$ are isomorphic
to the corresponding Dolbeault cohomology groups of $M_-.$ 
The class of compact homogeneous $CR$ manifolds to which this theory
applies includes the intersections of Matsuki dual orbits 
 in complex flag manifolds (cf. \cite{Mats88}).
\par 
The paper is organized as follows.
In the first section we discuss 
some basic facts
on compact homogeneous $CR$ manifolds, including Matsuki duality
and the notion of $\nt$-reductivity. The second section describes
$\Kf_0$-covariant fibrations $M_-\to{M}_0$ and,
under a special 
assumption on the partial complex structure
of $M_0$ (that we call {\HNR}), 
results on the cohomology, 
which
in part are contained in \cite{Marini} and 
have been further developed
in \cite{MaNa}. In a final section we discuss a simple example 
of $\nt$-reductive compact $CR$ manifolds which are intersections
of Matsuki dual orbits. 

\section{Compact homogeneous \texorpdfstring{$CR$}{TEXT} manifolds}
A \emph{compact homogeneous $CR$ manifold} is a $CR$ manifold $M$ on which a compact Lie group
$\Kf_0$ acts transitively as a group of $CR$ diffeomorphisms. Its $CR$ structure is uniquely determined
by the datum, for the choice of a \textit{base point} 
$p_0$ of $M$, of the $CR$ algebra $(\kt_0,\vt)$, where 
$\kt_0=\Lie(\Kf_0)$ and $\vt=d\piup^{-1}(T^{0,1}_{p_0}M),$ for the complexification $d\piup$ of the
differential of the canonical projection $\piup:\Kf_0\ni{x}\to{x}\cdot{p_0}\in{M}.$ We recall that, by the 
formal integrability
of the partial complex structure of $M,$ the subspace $\vt$ is in fact a Lie subalgebra of the complexification 
$\kt$ of $\kt_0.$ These pairs where introduced in \cite{MN05} to discuss homogeneous $CR$ manifolds
and the compact case was especially investigated in \cite{AMN2013}.
\subsection{Lie algebras of compact Lie groups} 
Compact Lie algebras are characterized in \cite{Bou82} by  
\begin{prop}
 For a real Lie algebra $\gtu$ the following are equivalent: 
\begin{enumerate}
 \item $\gtu$ is the Lie algebra of a compact Lie group;
 \item the analytic Lie subgroup $\Ind(\gtu)$ of $\GL_{\R}(\gtu)$ with Lie algebra $\ad(\gtu)$ is compact;
 \item on $\gtu$ a symmetric bilinear form can be defined which is invariant and positive definite;
 \item $\gtu$ is reductive, i.e. its adjoint representation is semi-simple and, for every $X\in\gtu$
 the endomorphism $\ad(X)$ is semisimple with purely imaginary eigenvalues;
 \item $\gtu$ is reductive with a negative semidefinite Killing form.\qed
\end{enumerate}
\end{prop}

\subsection{Complex flag manifolds} Let $\Gfu$ be a compact Lie group, with Lie algebra $\gtu$. The 
negative of the trace form of
a faithful representation of $\gtu$ yields an invariant scalar product $\kil$ on $\gtu$, that we use to
identify $\gtu$ with its
dual $\gtu^*$. In particular, the coadjoint orbits of $\Gfu$ are canonically
isomorphic to the adjoint orbits in $\gtu$. 
We follow \cite[\S{5.2}]{Kir}. Fix an element $H_0\in\ttu$ with $\ad(H_0)$ of maximal rank.
Then the commutant $$\ttu=\{H\in\gtu\mid [H_0,H]=0\}$$
is a maximal torus of $\gtu$. Denote by $\Tf_{H_0}=\{x\in\Gfu\mid\Ad(x)(H_0)=H_0\}$ 
the corresponding torus of $\Gfu$. The Weyl group $W_{H_0}$ is the quotient of the normalizer
$N_{\Gfu}(\ttu)=\{x\in\Gfu\mid\Ad(x)(\ttu)\subset\ttu\}=\{x\in\Gfu\mid [\Ad(x)(H_0),H_0]=0\}$ 
with respect to $\Tf_{H_0}$. \par 
We set $\gt=\C\otimes_{\R}\gtu$, 
$\hg_{\R}=i\ttu$ and denote by $\Rad$ the set of nonzero $\alpha\in\hg_{\R}^*$ such that 
$\gt^\alpha=\{X\in\gt\mid [H,X]=\alpha{X},\;\forall H\in\hg_{\R}\}\neq\{0\}$. 
We choose a system of simple roots $\alpha_1\hdots,\alpha_\ell$ in $\Rad$ 
with $\alpha_j(iH_0)>0$ for $j=1,\hdots,\ell$,
and denote by
\begin{equation*}
 C(H_0)=\{X\in\hg_{\R}\mid \alpha_i(X)>0,\;1\leq{i}\leq{\ell}\}
\end{equation*}
the corresponding positive Weyl chamber. 
\begin{lem}
 Every adjoint orbit $M$ in $\gtu$ intersects $iC(H_0)$
 in exactly one point. 
\end{lem} 
\begin{proof} Let $f(X)=\kil(X,H_0)$. Since $M$ is compact, $f$ has  stationary points on $M$.
A stationary point $X_0$ is characterized by  
\begin{equation*} 
 df(X_0)=0\Longleftrightarrow 0=\kil([X,X_0],H_0)=\kil(X,[X_0,H_0]),\; \forall X\in\gtu\Longleftrightarrow 
 X_0\in\ttu.
\end{equation*}
This shows that  the intersection $M\cap\ttu$ is not empty and all its
points are critical  for~$f$.  On the other hand, if $\Ad(x)(X_0)\in\ttu$, we can find $x'\in\Gfu$ with
$\Ad(x')(X_0)=\Ad(x)(X_0)$ and $\Ad(x)(\ttu)=\ttu$, so that the Weyl group is transitive on $M\cap\ttu$.
\end{proof}
\begin{prop}
 Let $\Gfu$ be a connected compact Lie group. Then: 
\begin{enumerate}
 \item For each $\Upsilon\in\gtu$ the stabilizer $\Sf(\Upsilon)=\{x\in\Gfu\mid \Ad(x)(\Upsilon)=\Upsilon\}$
 of $\Upsilon$ 
 contains a maximal torus of $\Gfu$.
 \item If $\Tf_0$ is any fixed
  maximal torus in $\Gfu$, then there are finitely many subgroups $\mathbf{A}_0$ with
 $\Tf_0\subseteq\mathbf{A}_0\subseteq\Gfu$.
 \item There is a unique maximal orbit $M$ 
 whose stabilizer at each point is a maximal torus of $\gtu$.
 \qed
\end{enumerate}
\end{prop}
\begin{defn}
 A \emph{flag manifolds} of $\Gfu$ is an orbit of its adjoint 
action on $\gtu$.
\end{defn}
\begin{thm}\label{flagthm}
 On a flag manifold $M$ of $\Gfu$ it is possible to define a complex structure and a $\Gfu$-invariant
 K\"ahler structure. \par 
 If $\Tf$ is a maximal torus of $\Gfu$, 
 the Weyl group $\Wg(\Tf,\Gfu)$ act transitively on the $\Gfu$-invariant complex structures of $M$. 
\end{thm} 
\begin{proof}
Fix a point $p_0$ of $M$, corresponding to $\Upsilon_0\in\gtu$. The stabilizer $\Sf_{\Gfu}(\Upsilon_0)$ contains
a maximal torus $\Tf$ and therefore there are finitely many parabolic subalgebras $\qt$ 
of the complexification $\gt$ of $\gtu$ with
$\qt\cap\gtu=\sta_{\gtu}(\Upsilon_0)$
(the Lie algebra of $\Sf_{\Gfu}(\Upsilon_0)$). Each $\qt$ corresponds to a different $\Gfu$-invariant complex structure
on $M$, and the Weyl group $\Wg(\Tf,\Gfu)$ act transitively on the $\qt$'s.
\end{proof}
Let $\Gf,$ with Lie algebra $\gt,$ be the complexification of $\Gf_u.$
The analytic subgroup $\Qf$ corresponding to the subalgebra $\qt$
in the proof of Theorem~\ref{flagthm} is parabolic and 
the complex structure of $M$ is also defined by 
its representation $M\simeq\Gf/\Qf$ as a complex homogeneous space.
Vice versa, if $\Qf$ is parabolic in $\Gf,$ the homogeneous space
$\Gf/\Qf$ is $\Gf_u$-diffeomorphic to a flag manifold of 
its compact form $\Gf_u.$ 
\subsection{Matsuki's Dual Orbits} \label{matsdual}
Let $\Gf_0$ be any real form 
and $\Gf_u$ a compact form of of a semisimple complex Lie group
$\Gf,$ with $\Kf_0=\Gf_0\cap\Gf_u$
a maximal compact subgroup of $\Gf_0$ and $\sigmaup,$
$\tauup$ the commuting
conjugations on $\Gf$, and on its Lie algebra
$\gt,$ with respect to  $\Gf_0$ and $\Gf_u,$
respectively.
The composition $\thetaup=\sigmaup\circ\tauup$ is the complexification
of a Cartan involution of $\Gf_0$ (and of its Lie algebra $\gt_0$)
commuting wiht
 $\sigmaup$ and $\tauup.$
The complexification $\Kf$ of $\Kf_0$ is the subgroup
$\Gf^{\thetaup}$ of the elements $x\in\Gf$ which are fixed by
$\thetaup.$ 
\par 
Let $\Qf$ be a parabolic subgroup of $\Gf$ and consider the actions
of the Lie groups
$\Gf_0,$ $\Kf,$ $\Kf_0$ on the complex flag manifold $M=\Gf/\Qf.$

\begin{ntz} For $p\in{N}=\Gf/\Qf$, let us set  
\begin{align}
\label{M+} 
M_+(p) & = \Gf_0\cdot{p}\simeq \Gf_0/\Ef_0,  
&&\text{with $\Ef_0=\{x\in\Gf_0\mid x\cdot{p}=p\},$}\\
\label{M-} M_-(p)&=\Kf\cdot{p} \simeq \Kf/\Vf,
&&\text{with $\Vf=\{z\in\Kf\mid z\cdot{p}=p\}$},\\
\label{M0} M_0(p)&=\Kf_0\cdot{p}=\Kf_0/\Vf_0, &
&\text{with $\Vf_0=\{x\in\Kf_0\mid x\cdot{p}=p\}$.}\end{align}
We note that $M_+(p)$ are  real, $M_-(p)$  complex and
$M_0(p)$  compact submanifolds of $M,$ and 
denote by $\Mf_+,\Mf_-,\Mf_0$ the sets of orbits of $\Gf_0,\Kf,\Kf_0$
in $M,$
respectively.
\end{ntz} 
We know (see \cite{Mats79, Mats88, Zie}) that 
\begin{thm}\label{numfinito}
There are  
finitely many orbits in $\Mf_+$ and in $\Mf_-.$
\par 
There is a one-to-one correspondence $\Mf_+\leftrightarrow\Mf_-$ 
such that $M_+(p_+)$ and $M_-(p_-)$ are related if and only if
$M_+(p_+)\cap{M}_-(p_-)=M_0(p_0)\in\Mf_0$ 
(Matsuki duality).
\end{thm}

The proof of Theorem~\ref{numfinito} is done by considering the
elements of $\Mf_{\pm}$ in the Grassmannian of $\dim(\qt)$-subspaces
of $\gt.$ On each orbit we can pick a $\qt'$ containining 
a $\thetaup$-stable Cartan subalgebra of $\gt.$ 
The fact that $\Mf_{\pm}$ is finite is then a consequence
of the fact that there are only finitely many congjugacy classes
of Cartan subalgebras with respect to the action of either
$\Gf_0$ or $\Kf$ (see e.g. \cite[Prop. 6.64]{Kn:2002}).
The last part of the statement is a consequence of the following
\begin {lem}\label{lemduality}
Let $\bt$ and $\bt'$ be Borel subalgebras containing $\theta$-stable Cartan subalgebras
of $\gt_0$.  If $\mathfrak{b}$ and $ \bt'$ are  either 
 $\Gf_0$-  or  $\Kf$-conjugate, 
 then they are  $\Kf_0$-conjugate.
\end{lem}
\begin{proof}
Let $\bt=\hg\oplus\nt$ and $\bt'=\hg'\oplus\nt'$, where $\hg$ and $\hg'$ are the complexifications 
of $\theta$-stable Cartan subalgebras $\hg_0$ and $\hg'_0$ of $\gt_0$ and 
$\nt,\,\nt'$ the nilradicals
of $\bt,\,\bt'$, respectively.  
\par 
Assume that $\bt=\Ad(g_0)(\bt')$, for some $g_0\in\Gf_0$. 
The proof 
in the case where $\bt=\Ad(k)(\bt')$ for some $k\in\Kf$ is similar, and will be omitted. 
\par 
Both
 $\Ad(g_0)(\hg'_0)$ and $\hg_0$ are Cartan subalgebras of $\bt\cap\gt_0,$ 
which is a solvable Lie subalgebra of $\gt_0$. Hence 
 $\Ad(g_0)(\hg_0')=\Ad(\exp(X_0))(\hg_0)$ for some $X_0$ in the nilradical
 of $\bt\cap\gt_0$. For $g_1=\exp(-X_0)\cdot{g}_0$, 
we have $\Ad(g_1)(\hg'_0)=\hg_0$ and $\Ad(g_1)(\bt')=\bt$. \par
 To show that $g_1\in\Kf_0,$ 
 we use the Cartan decomposition
 $\Gf_0=\exp(\pt_0)\cdot\Kf_0,$ 
where $\pt_0=\{Y\in\gt_0\mid \thetaup(Y)=-Y\},$  
 to write $g_1=\exp(Y_0)\cdot{k}_0$, with 
 $Y_0\in\pt_0$ and $k_0\in\Kf_0$. From 
\begin{equation*}
 \Ad(g_1)(\hg_0')=\hg_0=\theta(\hg_0)=\theta(\Ad(g_1)(\theta(\hg'_0))=\Ad(\theta(g_1))(\hg'_0)
\end{equation*}
 we obtain that 
\begin{equation*}
 \Ad(\exp(Y_0))(\Ad(k_0)(\hg'_0))=\Ad(\exp(-Y_0))(\Ad(k_0)(\hg'_0)),
\end{equation*}
i.e. $y_0=\Ad(\exp(2Y_0))$ normalizes the $\theta$-stable
 Cartan subalgebra $\Ad(k_0)(\hg'_0)$. This implies that $y_0\in\Kf_0$ and thus
 that $Y_0=0$,
 yielding $g_1=k_0\in\Kf_0$.
\end{proof}
From  Lemma~\ref{lemduality} we immediately obtain the statement on
the Matsuki duality in the case where $\Qf$ is a Borel subgroup
$\Bf.$ 
The general case follows by considering the natural fibration
$\Gf/\Bf\to\Gf/\Qf$ for a Borel $\Bf\subset\Qf.$ \qed
\par
\bigskip

\subsection{CR Manifolds}
Let $M_0$ be a smooth manifold of real dimensione $m$, countable at infinity.
A \textit{formally
integrable partial complex structure} of type $(n,k)$ 
(with $m=2n+k$) on $M_0$ 
is a pair 
$(HM_0,J),$ 
consisting of a rank $2n$  vector subbundle $HM_0$ of  
its 
tangent 
bundle 
$TM_0$ and a smooth 
fiber-preserving 
vector bundle isomorphism 
$
J:HM_0\rightarrow HM_0 
$
with $J^2=-I$, satisfying the integrability conditions 
\begin{equation*}\begin{cases}
[X,Y]-[JX,JY]\in\Gamma(M_0,HM_0),\\ [X,JY]+[JX,Y]=J([X,Y]-[JX,JY])
, 
\end{cases}\quad 
\forall X,Y\in\Gamma (M_0,HM_0).
\end{equation*}
This is equivalent to the fact that the subbundles 
$
T^{1,0}M_0=\{x-iJX\mid X\in HM_0\}
,$  
$T^{0,1}M_0=\{x+iJX\mid X\in HM_0\}$
of the complexification $\C\otimes{H}M_0$ 
of the structure bundle $HM_0$ are formally \textit{Frobenius integrable},
i.e. that 
\begin{align*}
&  [\Gamma(M_0,T^{1,0}M_0),
\Gamma(M_0,T^{1,0}M_0)]\subset\Gamma(M_0,T^{1,0}M_0),\quad
\intertext{or, equivalently, that} 
& [\Gamma(M_0,T^{0,1}M_0),\Gamma(M_0,T^{0,1}M_0)]
\subset\Gamma(M_0,T^{0,1}M_0).\end{align*}
These complex subbundles are the eigenspaces of $J$ corresponding to
the eigenvalues $\pm{i}$ and 
 $T^{1,0}M_0=\overline{T^{0,1}M_0}$,
 $T^{1,0}M_0\cap T^{0,1}M_0=\{0\}$,
$T^{1,0}M_0\oplus T^{0,1}M_0=\C\otimes HM_0$.

An (abstract) \textit{CR manifold} $M_0$ of type $(n,k)$ 
is a smooth paracompact real manifold
$M_0$ on which a
formally integrable partial complex structure $(HM_0,J),$
of type $(n,k),$ 
has been fixed. The integers $n$ and $k$ are
its $CR$ dimension and codimension, respectively.
Complex manifolds have $k=0,$ while for $n=0$ we say that
$M_0$ is \textit{totally real.}
\par
A smooth map $f:M_0\to{N}_0$ between $CR$ manifolds
is $CR$ iff its differential
${df}$ maps $HM_0$ to $HN_0$ and commutes with
the partial complex structures.\par 
If $f:M_0\to{N}_0$ is a $CR$ map and a smooth immersion (resp. 
embedding)
such that
$df^{-1}(HN_0)={H}M_0,$ then we say that $f$ is a $CR$ immersion
(resp. embedding).
Let $M_0$ be of type $(n_{M_0},k_{M_0})$ and $N_0$ of type
$(n_{N_0},k_{N_0}).$
A CR immersion (or embedding) $f:M_0\to N_0$ is \textit{generic}
if $n_{M_0}+k_{M_0}=n_{N_0}+k_{N_0}.$ 
\par 
The  
\textit{characteristic bundle}
$H^{0}M_0$ of a $CR$ manifold $M_0,$ of type $(n,k),$ is 
 the annihilator bundle 
 of its structure bundle $HM_0.$ It is a rank $k$ 
 linear subbundle of the real cotangent bundle $T^*M_0$. 
Its elements parametrize the \textit{scalar Levi forms} of $M_0$:
if $\xiup\in{H}^0_{p_0}M_0$ and $X\in{H}_{p_0}M_0$, then 
\begin{equation*}
\Levi_{\xiup}(X)=d\tilde{\xiup}(X,JX)=\xiup([J\tilde{X},\tilde{X}]),
\end{equation*} 
where $\tilde{\xiup}\in\Gamma(M_0.H^0M_0)$ extends $\xiup$ and
$\tilde{X}\in\Gamma(M_0,HM_0)$ extends $X$, is a quadratic form
on $H_{p_0}M$, which is Hermitian with respect to the complex structure
defined by $J.$ This $\Levi_{\xiup}$ is the \textit{scalar Levi form
at the characteristic $\xiup.$}\par 
If $\tilde{Z}=\tilde{X}+iJ\tilde{X}\in\Gamma(M_0,T^{0,1}M_0),$ 
then $\Levi_{\xiup}(X)=\tfrac{1}{2}\xiup(i[\tilde{Z}^*,
{\tilde{Z}}])$ and therefore we can as well consider the scalar
Levi forms as defined on $T^{0,1}_{p_0}M_0.$\par 
Let $\lambdaup(\alpha)=(\lambdaup^{+}(\xiup),\lambdaup^{-}(\xiup))$ 
be the signature of the Hermitian form $L_{\xiup}.$
Then $\nuup_{\xiup}=\min\{\lambdaup^{+}(\xiup),\lambdaup^{-}(\xiup)\}$
is its \textit{Witt index}.\par 
We say that $M_0$ is \emph{strongly $q$-pseudoconcave at a point $p_0$}
if the Witt index of $L_{\xiup}$ is 
greater or equal to ${q}$ for all nonzero
$\xiup\in{H}^0_{p_0}M_0.$\par 
For the relevant definitions of the tangential Cauchy-Riemann complex
and the relationship of its groups with $q$-pseudoconcavity 
we refer e.g. to \cite{AjHe1,AjHe2,14,BHN15,HN1995,40,78,23}.

\subsection{Homogeneous CR  Manifolds}
Let $\Gf_{0}$ be a real  Lie group with Lie algebra $\gt_0$, and denote by
 $\gt=\C\otimes\gt_0$  its complexification.
A $\Gf_{0}$-\textit{homogeneous \emph{CR} manifold}\index{CR!homogeneous manifold}  is a $\Gf_{0}$-homogeneous 
\textit{smooth} manifold endowed with a $\Gf_{0}$-invariant CR structure.
\par 
Let $M_0$ be a $\Gf_0$-homogeneous CR manifold. Fix a point 
$p_{0}\in{M}_0$ 
and denote by $\Ef_0$ 
its 
stabilizer 
in 
$\Gf_0$.
The natural projection
\begin{equation*}
 \pi:\Gf_0\longrightarrow \Gf/\Ef_0\simeq{M}_0,
\end{equation*}
makes $\Gf_0$ the total space of a principal $\Ef_0$-bundle with base ${M}_0$.
Denote by $\Zf(\Gf_0)$ 
the space of smooth sections of the pullback 
to $\Gf_{0}$ 
of  
$T^{0,1}M.$ 
It is the set of complex valued vector fields 
$Z$ on $\Gf_{0}$ such that 
$d\pi^{\C}(Z_{g})\in T_{\pi\left(g\right)}^{0,1}{M}_0,$ 
for all $g\in\Gf_{0}.$
\par 
Since $T^{0,1}M$ is formally integrable,
 $\Zf\left(\Gf_{0}\right)$
is formally integrable, i.e. 
$$
 [\Zf(\Gf_0),\Zf(\Gf_0)]\subset \Zf(\Gf_0).
$$
Being invariant by left translations, $\Zf(\Gf_0)$ 
is generated, as a  left $C^{\infty}(\Gf_0)$-module, 
by its left invariant vector fields. Hence 
\begin{equation}\label{CRalg}
\eg=(d\pi^{\C})^{-1}
(T_{x_{0}}^{0,1}M)\subset\mathfrak{g}=T_{e}^{\C}\Gf_{0}
\end{equation}
is an $\Ad(\Ef_0)$-invariant complex Lie subalgebra of $\mathfrak{g}$. 
We have:
\begin{lem}\label{lem2.1.1}
Denote by $\eg_0$ the Lie algebra of the 
isotropy subgroup $\Ef_0$.
 Then \eqref{CRalg} establishes a one-to-one correspondence 
 between the $\Gf_{0}$-homogeneus CR structures on ${M}_0=\Gf_{0}/\Ef_0$ 
 and the $\Ad_{\gt}(\Ef_0)$-invariant complex Lie subalgebras 
 $\eg$ of $\mathfrak{g}$ 
 such that $\eg\cap\gt_{0}=\eg_0.$
 \qed
\end{lem}
The pair $(\gt_0,\eg)$ completely determines the homogeneous $CR$
structure of $M_0$ and is called \emph{the $CR$-algebra}
of $M_0$ at $p_0$ 
(see \cite{MN05}).

\par\smallskip
Let $M_0$,
$N_0$ be $\Gf_0$-homogeneous CR manifolds and $\phiup:M_0\to{N}_0$ 
a $\Gf_0$-equi\-variant smooth map. 
Fix $p_0\in{M}_0$ al let 
$(\gt_0,\eg)$ and $(\gt_0,\mathfrak{f})$ be the CR algebras associated to $M_0$ at 
$p_0$ and to $N_0$ at $\phiup(p_0)$, respectively.  Then
$\eg\cap\overline{\eg}\subset\fg\cap\overline{\fg},$
and $\phiup$ is $CR$ if and only if 
$ \eg\subset\mathfrak{f}$
and 
is \textit{a CR-submersion} if and only if 
$\fg=\eg+\fg\cap\overline{\fg}.$
\par 
  The fibers of a $\Gf_0$-equivariant
 $CR$ submersion are homogeneous $CR$ manifolds: if $\Fb_0$
 is the stabilizer of $\phiup(p_0)\in N_0$, with Lie algebra $\fg_0=\fg\cap\gt_0$,
 then  
 $\Phi_0=\phi^{-1}(\phi(p_0))\simeq\Fb_0/\Ef_0$ has 
 $(\fg_0,\eg\cap
 \overline{\fg})$ as associated $CR$ algebra at $p_0$.
\begin{cor}
A CR-subbmersion $\phiup:M_0\to N_0$ has :
\begin{enumerate}
\item totally real fibers if and only if $\eg\cap\bar{\fg}=\bar\eg\cap\fg=\eg\cap\bar\eg$ ;
\item complex fibers if and only if $\eg\cap\bar{\fg}+\bar{\eg}
\cap\fg=\fg\cap\bar{\fg}$.\qed
\end{enumerate} 
\end{cor}

\subsection{\texorpdfstring{$\nt$}{TEXT}-reductive compact
\texorpdfstring{$CR$}{TEXT} manifolds}\label{nrid}
Let $\kt$ be a reductive complex Lie algebra 
and \par\centerline{
$
\kt=\mathfrak{z}\oplus\mathfrak{s}
$}\par\noindent
its decomposition 
into the sum of its center $\mathfrak{z}=\{X\in\kt\mid [X,Y]=0,\;\forall Y\in\kt\}$ 
and its semisimple ideal
$\mathfrak{s}=[\kt,\kt].$ We fix a 
 a faithful linear representation of $\kt$ in which 
the elements of $\zt$ correspond 
to semisimple matrices. We call semisimple and nilpotent the elements
of $\kt$  which are associated to semisimple and  
nilpotent matrices, respectively.
Each $X\in\kt$ admits a  unique Jordan-Chevalley decomposition
\begin{equation}
X=X_s+X_n,\;\;\text{with $X_s
,\; 
X_n
\in\kt$ 
and $X_s$ 
semisimple, $X_n
$ nilpotent.}
\end{equation}
\par
A real or complex  Lie subalgebra $\vt$ of $\kt$ is called \textit{splittable} if, for each $X\in\vt$, both $X_s$ and $X_n$ belong to $\vt$.
\par
Let $\vt$ be a Lie subalgebra of $\kt$ and $\mathrm{rad}(\vt)$ its radical (i.e. its maximal solvable ideal). 
Denote by 
\begin{equation}\label{vn}
\vt_n=\{ X\in\mathrm{rad}(\vt)
\mid 
\ad(X)\text{ is nilpotent}\}
\end{equation}
its nilradical (see \cite[p.58]{GOV94}).
It is the maximal  nilpotent ideal of $\vt$. 
We have (see \cite[Ch.VII,\S{5},Prop.7]{Bou75}):
\begin{prop}
Every splittable Lie subalgebra $\vt$ 
is a direct sum 
\begin{equation}
\vt=\vt_r+\vt_n,
\end{equation}
of its nilradical $\vt_n$ and of a reductive subalgebra $\vt_r$,  
which is uniquely determined modulo conjugation by elementary automorphisms of $\vt$.\qed
\end{prop}
\par\smallskip
We assume in the following that $\kt$ is the complexification of a compact Lie algebra $\kt_0$. 
Since compact Lie algebras are reductive, and the complexification of a reductive real Lie algebra is
reductive, $\kt$ is complex reductive. Conjugation in $\kt$ will be taken with respect to the
real form $\kt_0$.

\begin{prop}
For any complex Lie subalgebra $\vt$ of $\kt$, the intersection $\vt\cap\bar\vt$ is reductive and splittable.
 In particular  $\vt\cap\bar\vt\cap\vt_n=\{0\}$. A splittable $\vt$ admits a Levi-Chevalley decomposition with a reductive Levi factor containing $\vt\cap\bar\vt.$ \qed
\end{prop}
Let $M_0$ be a $\Kf_0$-homogeneous $CR$ manifold,
$p_0\in{M}_0$  and $(\kt_0,\vt)$ its $CR$ algebra at 
$p_0.$ 
\begin{dfn}\label{vriduttiva}
 We say that $
M_0
$ 
and its CR-algebra $(\kt_0,\vt)$ are $\nt$-\textit{reductive}\index{CR!reductive algebra} if 
\begin{equation}
\vt=(\vt\cap\overline{\vt})\oplus\vt_n,
\end{equation}
i.e. if $\vt_r=\vt\cap\bar\vt$ is a reductive 
Levi factor in $\vt$.
\end{dfn} 
If $(\kt_0,\vt)$ is $\nt$-reductive, then $\vt$ is splittable. Indeed the elements of $\vt_n$ are nilpotent and 
those of $\vt\cap\kt_0$ semisimple.
Having a set of generators that are either nilpotent or semisimple, $\vt$
is splittable (see \cite[
Ch.VII,\S{5},Thm.1]{Bou75}). 
\par 
Let us consider the situation of  \S\ref{matsdual} and keep the notation therein.
The submanifolds $M_0(p)$ are examples of 
compact $\Kf_0$-homogeneous $CR$ submanifold of 
the complex flag manifolds $M.$ In \cite[\S{6}]{AMN2013} it was shown that
\begin{prop} The inclusion $M_0(p)\hookrightarrow{M}_-(p)$ is a generic $CR$ embedding.
A necessary and sufficient condition for $M_0(p)$ to be $\nt$-reductive is that 
\begin{equation*}
M_0(p)=M_+(p)\cap{M}_-(p). 
\end{equation*}
\end{prop}
\section{Mostow fibration and applications to cohomology} \label{most}
We use the notation of \S\ref{nrid}. Let $\Vf$ be the analytic subgroup of $\Kf$ with Lie algebra $\vt$
and $\Vf_0$ the isotropy subgroup at $p_0\in{M}_0$, having Lie algebra $\vt_0=\vt\cap\kt_0.$
\begin{prop} [{\cite[Theorem 26]{AMN2013}}]
If $M_0$ is $\nt$-reductive, then $\Vf$ is an algebraic subgroup of $\Kf$ and the natural map
\begin{equation*}
 M_0=\Kf_0/\Vf_0 \longrightarrow M_-=\Kf/\Vf
\end{equation*}
is a generic $CR$ embedding. \qed
\end{prop}
When, as in \S\ref{matsdual}, $M_0$ is the intersection of two Matsuki dual orbits in a complex flag manifold,
 $M_-$ has a compactification $\tilde{M}_-$ which is a 
complex projective variety. Thus, in principle, we can study its Dolbeault cohomology by 
algebraic geometric techniques. On the other hand, by 
Mostow's decomposition (see \cite{Most56,Most62})
we know that $M_-$ is a $\Kf_0$-equivariant fiber bundle over $M_0,$
whose fibers are totally real Euclidean subspaces. 
We can exploit this fact for constructing an exhaustion function on $M_-$ whose level sets
can be used to relate the tangential $CR$ cohomology of $M_0$ to the Dolbeaut cohomology
of $M_-.$
\subsection{Mostow fibration of \texorpdfstring{$M_-$}{TEXT} (I)}
The isotropy $\Vf_0$ is a maximal compact subgroup of $\Vf$ and, putting together the
Levi-Chevalley decomposition of $\Vf$ and the Cartan decomposition of $\Vf_r$,  
we have a diffeomorphism
\begin{equation}
 \Vf_0\times\vt_0\times\vt_n\ni (x,Y,Z)\longrightarrow x\cdot\exp(iY)\cdot\exp(Z)\in\Vf.
\end{equation}
Then by \cite[Theorem A]{Most62} \textsl{we can find a closed Euclidean subspace $F$ of $\Kf$ 
such that $\ad(y)(F)=F,$ for all $y\in\Vf_0$, and 
\begin{equation}
 \Kf_0\times{F}\times\vt_0\times\vt_n\ni (x,f,Y,Z)\longrightarrow x\cdot{f}\cdot\exp(iY)\cdot\exp(Z)
 \in\Kf
\end{equation}
is a diffeomorphism.} \par 
Let $\kil$ be an $\Ad(\Kf_0)$-invariant scalar product on $\kt_0$ and set 
\begin{equation*}
 \mt_0=((\vt+\overline{\vt})\cap\kt_0)^\perp,\;\; \mt=\C\otimes\mt_0.
\end{equation*}
%
Since\footnote{In fact, if $\bar{Z}\in\bar{\vt}_n$, then $\bar{Z}=-Z+(Z+\bar{Z})$, with $Z=\Bar{\bar{Z}}\in\vt_n$
and $Z+\bar{Z}\in\kt_0$. The sum is direct because $\vt_n\cap\kt_0=\{0\}$.}
$
 \vt_n\oplus\bar{\vt}_n=\vt_n\oplus  ([\vt_n\oplus\bar{\vt}_n]\cap\kt_0),
$ we obtain the decomposition
\begin{equation*}
 \kt=\kt_0 \oplus\,{i}\,\mt_0\oplus (i\,\vt_0\oplus\vt_n).
\end{equation*}
which suggests that 
$\exp(i\mt_0)$ could be a reasonable candidate for the fiber $F$ of the Mostow fibration.
Let us consider the smooth map
\begin{equation}\label{eq3.5a}
 \Kf_0\times\mt_0\times \,\vt_0\times{v}_n\ni (x,T,Y,Z)\xrightarrow{\;\;\;\;\;} x\cdot\exp(iT)\cdot
 \exp(iY)\cdot \exp(Z)\in\Kf.
\end{equation}
\begin{prop}\label{prop2.2}
 The map \eqref{eq3.5a} is onto and we can find $r>0$ such that its restriction to
 $\{\kil(T,T)<r^2\}$ 
 is a diffeomorpism with the image. \qed
\end{prop} 
Let  $\Kf_0\times_{\Vf_0}\mt_0$ be the quotient of $\Kf_0\times\mt_0$ by
$(x_1,T_1)\sim (x_2,T_2)$ iff $x_2=x_1\cdot{y}$ and $T_1=\Ad(y)(T_2)$  
and $\piup:\Kf \ni z \to z\cdot\Vf\in{M}_-$ the canonical projection. 
By passing to the quotients, \eqref{eq3.5a} yields a smooth map 
\begin{equation}\label{eq3.5b}
 \Kf_0\times_{\Vf_0}\mt_0 \ni [x,T] \longrightarrow \piup(x\cdot \exp(iT)) \in {M}_-,
\end{equation}
which is surjective and, when restricted to $\{\kil(T,T)<r^2\},$ defines a tubular neighborhood 
$U_r$ of
$M_0$ in $M_-.$ The function $\phiup([x,T])=\kil(T,T)/(r^2-\kil(T,T))$ is then an exhaustion 
function of the tubular neighborhood $U_r.$
\subsection{A local result} 
We can use Proposition~\ref{prop2.2} to precise, in this special case, the size of the
tubular neighborhoods of  
\cite[Theorem 2.1]{HN1995}:
\begin{prop}
 Assume that $M_0$ is $q$-pseudoconcave, of type $(n,k)$. Then we can find $r_0>0$  such that
 $U_r$ is is $q$-pseudoconcave and $(n-q)$-pseudoconvex and
 the natural restriction maps
\begin{equation*}
 H^{p,j}_{\bar{\partial}}(U_r)\longrightarrow H^{p,j}_{\bar{\partial}_{M_0}}(M_0) 
\end{equation*}
are isomorphisms of finite dimensional vector spaces 
for $0<r<r_0,$ for all \mbox{$0\leq{p}\leq{n+k}$}
and either $j<q$ or $j>n-q.$ \qed
\end{prop}
\subsection{Mostow fibration of \texorpdfstring{$M_-$}{TEXT} 
in the \texorpdfstring{\HNR}{TEXT} case (II)}
In  \cite{MaNa} the Authors show that \eqref{eq3.5a} and \eqref{eq3.5b} are not,
in general, global diffeomorphisms. To  single out a large class of
of compact homogeneous $\nt$-reductive $CR$ manifolds having a \textit{nice}
Mostow fibration, they
introduce the following notion\footnote{Actually they consider 
a slightly less restrictive condition, which is related to a notion of weak $CR$-degeneracy
for homogeneous $CR$ manifolds that was introduced in \cite{MN05}.}. 
\begin{defn}
 The $CR$ algebra $(\kt_0,\vt)$ is {\HNR} if $\vt=(\vt\cap\overline{\vt})\oplus\vt_n$ and\par
 \centerline{$\qt=\{Z\in\kt\mid [Z,\vt_n]\subset\vt_n\}$ \;\;\;  is parabolic.}
\end{defn}
Note that, when $(\kt_0,\vt)$ is $\nt$-reductive, it is always possible to find a parabolic
$\qt$ in $\kt$ with $\vt\subset\qt,$ $\qt=(\qt\cap\bar{\qt})\oplus\qt_n,$
and $\vt_n\subset\qt_n.$ Then $(\kt_0,\vt_r\oplus\qt_n)$ is {\HNR} and describes a
\textit{stronger} $\Kf_0$-homogeneous $CR$ structure on the same manifold $M_0.$
\begin{prop}
 If its $CR$ algebra $(\kt_0,\vt)$ is {\HNR}, then \eqref{eq3.5a} and \eqref{eq3.5b} 
 are diffeomorphisms.
\end{prop}
In this case $M_-$ admits \textit{a Mostow fibration with Hermitian fiber.} 
\subsection{Application to cohomology}
In the {\HNR} case we can use $\phiup([x,T])=\kil(T,T)$ as an exhaustion function on $M_-$.
Assume that $M_0$ has type $(n,k).$ 
It is shown in \cite{MaNa} that, for $T\in\mt_0$ and $\xiup=\kil(T,\,\cdot\,)\in{H}_{p_0}^0M_0,$ 
if the Levi form $\Levi_{\xiup}$ has signature $(\lambdaup^+(\xiup),\lambdaup^-(\xiup))$, then
the complex Hessian of $\phiup$ at the point $\piup(x\cdot\exp(iT))$ has signature
$(\lambdaup^+(\xiup)+k,\lambdaup^-(\xiup)).$ 
Then we get 
\begin{prop}
 Assume that $M_0$ is $q$-pseudoconcave and
 has a $CR$ algebra which is {\HNR}.
Then
 $M_-$ is $q$-pseudoconcave and $(n-q)$-pseudoconvex
  and the natural restriction maps
\begin{equation*}
 H^{p,j}_{\bar{\partial}}(U_r)\longrightarrow H^{p,j}_{\bar{\partial}_{M_0}}(M_0) 
\end{equation*}
are isomorphisms of finite dimensional vector spaces 
for $0<r<r_0,$ for all \mbox{$0\leq{p}\leq{n+k}$} 
and either $j<q$ or $j>n-q.$ 
\end{prop} 
\begin{proof}
 The statement follows from \cite{AG62} and the computation of the signature of the exhaustion function
 $\phiup.$
\end{proof}

\section{Example:
General orbits of \texorpdfstring{$\SU(p,q)$}{TEXT} in the Grassmannian}
Let us
consider the orbits of the real form 
$\Gf_0=\SU(p,q)$ of $\Gf=\SL_{p+q}(\C)$ in $\Gr_{\! m}(\C^{p+q})$. 
We assume that ${p}\leq{q}$. Let $\hf$ 
 be the Hermitian symmetric form of signature $(p,q)$
in $\C^{p+q}$ employed to define $\SU(p,q)$. \par 
The orbits of $\Gf_0$ 
are classified by the signature of the restriction of $\hf$ 
to their $m$-planes: to a pair of nonnegative integers $a,b$ with
\begin{equation}\tag{$*$}\label{starex}
a+b\leq{m},\;\; p_0=\max\{0,m-q\}\leq{a}\leq{p},\;\;
q_0=\max\{0,m-p\}\leq{b}\leq{q},\end{equation}
correspond the orbit $M_+(a,b)$ consisting of $m$-planes $\ell$  
for which  
 $\ker({\hf }|_{\ell})$ has signature $(a,b).$ 
\par 
To fix a
maximal compact subgroup of $\SU(p,q)$ we choose 
a couple of $\hf$-orthogonal subspaces $W_+\simeq\C^p,$ 
$W_-\simeq\C^q$ of $\C^{p+q},$ 
with $\hf>0$ on $W_+,$ $\hf<0$ on $W_-,$
$\C^{p+q}=W_+\oplus{W}_-.$ 
Then 
$\Kf_0\simeq\mathbf{S}(\Ub(p)\times\Ub(q))$,  and
$\Kf=\C\otimes\Kf_0\simeq\mathbf{S}(\GL_{p}(\C)\times\GL_q(\C)),$
with the first factor operating on $W_+$ and the second on
$W_-$. The orbits
of $\Kf$ in $\Gr_{\! m}(\C^{p+q})$ are characterized by 
the dimension of $\ell\cap{W}_+$ and $\ell\cap{W}_-.$ 
Let us set, for $a,b$ satisfying \eqref{starex}, 
\begin{align*}
 M_+(a,b)&
 =\{\ell\in\Gr_{\! m}(\C^{p+q})\mid \hf|_\ell \;\;\text{has signature $(a,b)$}\},\\
 M_-(a,b)&=\{\ell\in\Gr_m(\C^{p+q})\mid \dim_{\C}(\ell\cap{W_+})=a,\;\; \dim_{\C}(\ell\cap{W_-})=b\},\\
 M_0(a,b)&=M_+(a,b)\cap{M_-(a,b)}.
\end{align*}
To describe the $CR$ structure of $M_0(a,b)$ it suffices to compute
the Lie algebra of the stabilizer in $\Kf$ of any of its points.
Let $e_1,\hdots,e_p$ be an 
orthonormal basis of $W_+$ and $e_{p+1},\hdots,e_{p+q}$ an
orthonormal basis of $W_-.$ 
Set $c=m-a-b$ and $n_1=a,$ $n_2=c$, $n_3=p-a-c,$ $n_4=b,$ $n_5=c,$
$n_6=q-b-c.$
 Let us choose the base point 
$p_0=\langle e_1,\hdots,e_a, e_{a+1}+e_{p+b+1},\hdots, e_{a+c},
e_{p+a+c},
e_{p+1},\hdots,e_{p+b}\rangle\in{M}_0(a,b).$
Then \begin{equation*}
\vt=\left.\left\{ \begin{pmatrix}
Z_{1,1} & 0 & Z_{1,3}& 0 & 0 & 0\\
0 & Z_{2,2} & Z_{2,3}& 0 & 0 & 0 \\
0 & 0 & Z_{3,3}& 0 & 0 & 0\\
0 & 0 & 0 & Z_{4,4} & 0 & Z_{4,6}\\
0 & 0 & 0 & 0 & Z_{2,2} & Z_{5,6} \\
0 & 0 & 0 & 0 & 0 & Z_{6,6}
\end{pmatrix} \right| Z_{i,j}\in\C^{n_i\times{n}_j}
\right\}\cap\slt_{p+q}(\C).
\end{equation*}
We see that $(\kt_0,\vt)$ is {\HNR} and 
$M_0(a,b)$ 
has $CR$ dimension \par\centerline
{$n=n_1n_3+n_2n_3+n_4n_6+n_2n_6$.}\par 
 We have 
\begin{equation*}
\mt=\left.\left\{ \begin{pmatrix}
0 & W_{1,2} & 0 & 0 & 0 \\
W_{2,1}& W_{2,2} & 0 & 0 & 0 & 0 \\
0 & 0 & 0 & 0 & 0 & 0 \\
0 & 0 & 0 & 0 & W_{4,5} & 0 \\
0 & 0 & 0 & W_{5,4} & -W_{2,2} & 0 \\
0 & 0 & 0 & 0 & 0 & 0
\end{pmatrix}
\right| W_{i,j}\in\C^{n_i\times{n}_j}
\right\}\cap\slt_{p+q}(\C)
\end{equation*}
and hence the $CR$ codimension of $M_0(a,b)$ is \par\centerline
{$ k= n_1n_2+n_2n_4+n_2^2=n_2(n_1+n_2+n_4).$} \par 
We observe that for $c=0$ the $M_+(a,b)$ are the open orbits,
while the minimal orbit of $\Gf_0$ is $M_+(p_0,q_0).$
Then $M_0(a,b)$ is $\muup$-pseudoconcave with\par\centerline 
{$\muup=\min\{p-a-c,q-b-c\},$}\par\noindent
and  $M_-(a,b)$ is $\muup$-pseudoconcave and
$(n-\muup)$-pseudoconvex.

\providecommand{\bysame}{\leavevmode\hbox to3em{\hrulefill}\thinspace}
\providecommand{\MR}{\relax\ifhmode\unskip\space\fi MR }
\providecommand{\MRhref}[2]{%
  \href{http://www.ams.org/mathscinet-getitem?mr=#1}{#2}
}
\providecommand{\href}[2]{#2}

\end{document}